\documentclass[12pt,reqno]{amsart}
\usepackage{enumitem}
\usepackage{graphicx, tikz, tkz-graph}
\usepackage{wrapfig, float,subfigure,  caption}
\usepackage{amsmath,amssymb}
\usepackage{mathrsfs}
\usepackage{amsfonts}
\usepackage{amssymb,amsmath}
\usepackage{amsthm}
\usepackage{graphicx}
\usepackage{cite}
\usepackage{lineno}
\usepackage{enumitem}
\usepackage{xcolor}

\usepackage{caption}
\usetikzlibrary{graphs}
\usetikzlibrary{graphs.standard, quotes, arrows, positioning, calc, arrows.meta}
\usepackage{pgfplots}
\usepackage{xcolor}
\pgfplotsset{compat=newest}
\usepgfplotslibrary{polar}
\usepgflibrary{shapes.geometric}
\usetikzlibrary{calc}
\pgfplotsset{my style/.append style={
               axis x line=middle, 
               axis y line= middle,
               xlabel={$x$}, 
               ylabel={$y$},
               legend pos = south east,
               label style={font=\footnotesize},
               tick label style={font=\tiny}}}
\usepgfplotslibrary{fillbetween}
\usetikzlibrary{patterns}
\usepackage{harpoon}
\newtheorem{theorem}{Theorem}[section]
\newtheorem{lemma}[theorem]{Lemma}
\newtheorem{corollary}[theorem]{Corollary}

\newtheorem*{theorem*}{Theorem}
\theoremstyle{definition}

\theoremstyle{remark}

\definecolor{mygray}{RGB}{105, 105, 105}
\numberwithin{equation}{section}
\newcommand{\del}{\backslash}

\newcommand{\C}{\mathcal{C}}

\makeatletter
\pgfdeclarepatternformonly[\GridSize]{MyGrid}{\pgfqpoint{-1pt}{-1pt}}{\pgfqpoint{\GridSize}{\GridSize}}{\pgfqpoint{\GridSize}{\GridSize}}
{
     \pgfsetcolor{\tikz@pattern@color}
     \pgfsetlinewidth{0.3pt}
     \pgfpathmoveto{\pgfqpoint{0pt}{0pt}}
     \pgfpathlineto{\pgfqpoint{0pt}{\GridSize + 0.1pt}}
     \pgfpathmoveto{\pgfqpoint{0pt}{0pt}}
     \pgfpathlineto{\pgfqpoint{\GridSize + 0.1pt}{0pt}}
     \pgfusepath{stroke}
}
\makeatother

\newdimen\GridSize
\tikzset{
    GridSize/.code={\GridSize=#1},
    GridSize=6pt
}

\makeatletter
\pgfdeclarepatternformonly[\LineSpace]{my north west lines}{\pgfqpoint{-1pt}{-1pt}}{\pgfqpoint{\LineSpace}{\LineSpace}}{\pgfqpoint{\LineSpace}{\LineSpace}}
{
    \pgfsetcolor{\tikz@pattern@color}
    \pgfsetlinewidth{0.4pt}
    \pgfpathmoveto{\pgfqpoint{0pt}{\LineSpace}}
    \pgfpathlineto{\pgfqpoint{\LineSpace + 0.1pt}{-0.1pt}}
    \pgfusepath{stroke}
}
\makeatother
\newdimen\LineSpace
\tikzset{
    line space/.code={\LineSpace=#1},
    line space=6.5pt
}
\begin{document}

\title{The symmetric strong circuit elimination property}

\author{Christine Cho}
\address{Mathematics Department\\
  Louisiana State University\\
  Baton Rouge, Louisiana}
\email{ccho3@lsu.edu}

\author{James Oxley}
\address{Mathematics Department\\
  Louisiana State University\\
  Baton Rouge, Louisiana}
\email{oxley@math.lsu.edu}

\author{Suijie Wang}
\address{School of Mathematics\\
  Hunan University\\
  Changsha 410082, Hunan, P. R. China}
\email{wangsuijie@hnu.edu.cn}

\date{\today}

\begin{abstract}
    If $C_1$ and $C_2$ are circuits in a matroid $M$ with $e_1$ in $C_1-C_2$ and $e$ in $C_1\cap C_2$, then $M$ has a circuit $C_3$ such that $e\in C_3\subseteq (C_1\cup C_2)-e$. 
    This strong circuit elimination axiom is inherently asymmetric. 
    A matroid $M$ has the symmetric strong circuit elimination property (SSCE) if, when the above conditions hold and $e_2\in C_2-C_1$, there is a circuit $C_3'$ with $\{e_1,e_2\}\subseteq C_3'\subseteq (C_1\cup C_2)-e$.
    We prove that a  connected matroid has this property if and only if it has no two skew circuits.
    We also characterize such matroids in terms of forbidden series minors, and we give a new matroid axiom system that is built around a modification of SSCE. 
\end{abstract}
\maketitle

\section{Introduction}
A matroid $M$ has the \textit{symmetric strong circuit elimination property} (SSCE) if, whenever $C_1$ and $C_2$ are circuits and $e_1,e_2,$ and $e$ are elements such that $e_1\in C_1-C_2$, $e_2\in C_2-C_1$, and $e\in C_1\cap C_2$, there is a circuit $C_3$ that contains $\{e_1,e_2\}$ and is contained in $(C_1\cup C_2)-e$. 
Sets $X$ and $Y$ in a matroid are \textit{skew} if $r(X)+r(Y)=r(X\cup Y)$. 
The next theorem, the main result of this paper, gives several characterizations of matroids satisfying SSCE. 
A matroid $M$ is \textit{unbreakable} if $M$ is connected and $M/F$ is connected for every flat $F$ of $M$.

\begin{theorem}\label{main}
    The following are equivalent for a connected matroid $M$.    
    \begin{enumerate}[label=(\roman*)]
        \item $M$ has the symmetric strong circuit elimination property; 
        \item $M$ has no pair of skew circuits;
        \item for all integers $k$ and $l$ exceeding two, $M$ has no series minor isomorphic to $S(U_{k-2,k},U_{l-2,l})$; and
        \item $M^*$ is unbreakable. 
    \end{enumerate}
\end{theorem}

To see that not every connected matroid has the symmetric strong circuit elimination property, consider the matroid $N_5$ that is obtained from a $3$-circuit $\{e_1, e_2,e\}$ by adding $f_i$ in parallel to $e_i$ for each $i$ in $\{1,2\}$. 
Consider the circuits $C_1=\{e_1, f_2,e\}$ and $C_2=\{e_2, f_1,e\}$. 
Then $e_1\in C_1-C_2$ and $e_2\in C_2-C_1$, but $N_5$ has no circuit contained in $(C_1 \cup C_2) - e$ that contains $\{e_1,e_2\}$. 
Clearly $N_5$ has $\{e_1,f_1\}$ and $\{e_2,f_2\}$ as skew circuits.

The equivalence between (ii) and (iii) in Theorem \ref{main} extends a result of Drummond, Fife, Grace, and Oxley \cite[Proposition 15]{cd}, which shows that a connected binary matroid $M$ has a pair of skew circuits if and only if $M$ has a series minor isomorphic to $N_5$. 
Note that $N_5$ is isomorphic to the series connection of two copies of $U_{1,3}$.   
A matroid $M$ is \textit{circuit-difference} if $C_1\Delta C_2$ is a circuit for every distinct intersecting pair of circuits $C_1$ and $C_2$ of $M$.   
Drummond et al. \cite[Theorem 1]{cd} proved the following. 

\begin{theorem}\label{regcd}
    A connected regular matroid is circuit difference if and only if it has no pair of skew circuits. 
\end{theorem}

\noindent Combining Theorems \ref{main} and \ref{regcd}, we immediately obtain the following. 

\begin{corollary}
    A connected regular matroid has the symmetric strong circuit elimination property if and only if $M$ is circuit-difference. 
\end{corollary}

Oxley and Pfeil \cite[Theorem 1.1]{ub} proved that a loopless matroid $M$ is unbreakable if and only if $M^*$ has no pair of skew circuits. 
Thus (ii) and (iv) in Theorem \ref{main} are equivalent.
Oxley and Pfeil gave several other characterizations of connected matroids with no skew circuits. Thus the list of equivalent statements in Theorem \ref{main} could be extended. 

Theorem \ref{main} gives several characterizations of when a connected matroid does not have a pair of skew circuits. 
The next theorem characterizes when a connected binary matroid has three skew circuits where, for any integer $k$ exceeding one, a matroid $M$ has $k$ \textit{skew circuits} if $M$ has a set $\{C_1,C_2,\ldots,C_k\}$ of $k$ circuits  such that 
        \[
            M\big|\left(\cup_{i=1}^k C_i\right)=(M|C_1)\oplus (M|C_2)\oplus\ldots\oplus (M|C_k).
        \]

\begin{theorem}\label{3thm}
    Let $M$ be a connected binary matroid with three skew circuits. Then $M$ has a series minor isomorphic to $M(G)$, where $G$ is one of the graphs shown in Figure \ref{3sk}.  
\end{theorem}

\begin{figure}[ht]
    \centering
\begin{tikzpicture}[scale=1]
    \draw[thick] (0,1)to[out=240,in=120] (0,0)to[out=30, in=150] (1,0)to[out=60, in=300](1,1);
    \draw[thick] (0,1)to[out=300,in=60] (0,0)to[out=330, in=210] (1,0)to[out=120, in=240](1,1);
    \draw[thick] (0,1) -- (1,1);

    \draw[fill=black] (0,0) circle (1.9pt); 
    \draw[fill=black] (1,0) circle (1.9pt); 
    \draw[fill=black] (0,1) circle (1.9pt);
    \draw[fill=black] (1,1) circle (1.9pt); 
    \draw (.5,-0.5) node {$L_1$};

    \begin{scope}[shift={(2.7,.3)},rotate={180},scale=.8]
    \draw[thick] (0,0)to[out=180,in=290] (-.866,.5)to[out=0, in=120] (0,0)to[out=60, in=180](.866,.5)to[out=240,in=0](0,0)to[out=240,in=120](0,-1)to [out=60,in=300](0,0);
    \draw[thick] (-.866,0.5)to[out=270,in=150] (0,-1);
    \draw[thick] (.866,0.5)to[out=270,in=30] (0,-1);

    \draw[fill=black] (0,0) circle (2.3pt); 
    \draw[fill=black] (.866,.5) circle (2.3pt); 
    \draw[fill=black] (-.866,.5) circle (2.3pt); 
    \draw[fill=black] (0,-1) circle (2.3pt);
    \draw (0,0.95) node {$L_2$};
    
    \end{scope}

    \begin{scope}[shift={(5.2,0)},scale=.7]
    \draw[thick] (-.75,0)--(.75,0)--(0,1.5)--(-.75,0);
    \draw[thick] (-1.25,1.25)--(0,1.5)--(1.25,1.25);
    \draw[thick] (-.75,0)to[out=150,in=270] (-1.25,1.25)to[out=330, in=90] (-.75,0);
    \draw[thick] (.75,0)to[out=30,in=270] (1.25,1.25)to[out=210, in=90] (.75,0);
    \draw[fill=black] (-.75,0) circle     (2.6pt); 
    \draw[fill=black] (.75,0) circle      (2.6pt); 
    \draw[fill=black] (0,1.5) circle      (2.6pt); 
    \draw[fill=black] (-1.25,1.25) circle (2.6pt);
    \draw[fill=black] (1.25,1.25) circle  (2.6pt);
    \draw (0,-0.65) node {$L_3$};
    
\end{scope}

    \begin{scope}[shift={(8.5,0)},scale=.85]
    \draw[thick] (-1,0)--(0,0)--(-1,1)--(0,1)--(-1,0);
    \draw[thick] (-1.75,0.75)--(-1,1);
    \draw[thick] (.75,.75)--(0,1);
    \draw[thick] (-1,0)to[out=160,in=280] (-1.75,.75)to[out=350, in=110] (-1,0);
    \draw[thick] (0,0)to[out=20,in=260] (.75,.75)to[out=190, in=70] (0,0);

    \draw[fill=black] (-1.75,0.75) circle     (2.3pt); 
    \draw[fill=black] (.75,.75) circle        (2.3pt); 
    \draw[fill=black] (-1,0) circle           (2.3pt); 
    \draw[fill=black] (-1,0) circle           (2.3pt); 
    \draw[fill=black] (0,0) circle            (2.3pt); 
    \draw[fill=black] (-1,1) circle           (2.3pt); 
    \draw[fill=black] (0,1) circle            (2.3pt);
    \draw (-0.5,-0.5) node {$L_4$};
\end{scope}
  \begin{scope}[shift={(10.2,0)}, scale=1.1] 
   \coordinate (n) at (0.5,1.25); \coordinate (e) at (1.25,.75);  \coordinate (s) at (0.5,0); \coordinate (w) at (-0.25,.75); 
   \coordinate (m) at (0.5,0.75);
    \draw[thick] (n)--(e); \draw[thick] (n)--(m); \draw[thick] (n)--(w); 
    \draw[thick] (s)to [bend left=23] (e); \draw[thick] (s)to [bend right=23] (e);
    \draw[thick] (s)to [bend left=23] (w); \draw[thick] (s)to [bend right=23] (w);
    \draw[thick] (s)to [bend left=25] (m); \draw[thick] (s)to [bend right=25] (m);
    
    \draw[fill=black] (n) circle (1.8pt); 
    \draw[fill=black] (e) circle (1.8pt);
    \draw[fill=black] (m) circle (1.8pt);
    \draw[fill=black] (s) circle (1.8pt);
    \draw[fill=black] (w) circle (1.8pt); 
    \draw (.5,-0.4) node {$L_5$};
\end{scope}

\end{tikzpicture}
\caption{ }\label{3sk} 
\end{figure}

After a section of preliminaries, the main theorem is proved in Section 3. 
Section 4 proves Theorem \ref{3thm}.
Finally, in Section 5, we give a new circuit axiom system for matroids that is built around a modification of SSCE. 

\section{Preliminaries}
The terminology and notation used here will follow \cite{ox}. 
The class of matroids satisfying SSCE is not closed under taking contractions.
To see this, recall that $N_5$ has a pair of skew circuits. 
Although $M(K_4)$ has no pair of skew circuits, every single-element contraction of $M(K_4)$ is isomorphic to $N_5$.
We now show that the class of matroids satisfying SSCE is closed under series contraction. 

\begin{lemma}\label{closed}
    The class of matroids satisfying the symmetric strong circuit elimination property is closed under series minors.  
\end{lemma}
    \begin{proof}
       Clearly the class of matroids satisfying SSCE is closed under deletion. 
       Now suppose that $M$ satisfies SSCE and has $\{f,g\}$ as a cocircuit. 
       Then, for every circuit $C$ of $M$, either $C$ or $C-g$ is a circuit of $M/g$.  
       Let $C_1$ and $C_2$ be circuits of $M/g$ with $e$ in $C_1\cap C_2$ such that $e_1\in C_1-C_2$ and $e_2\in C_2-C_1$. 
       Since $\{f,g\}$ is a cocircuit of $M$, we have $\C(M/g)=\mathcal{D}_1\cup \mathcal{D}_2$ where 
       $\mathcal{D}_1=\{C\in \C(M):\{f,g\}\cap C=~\emptyset\}$ and $\mathcal{D}_2=\{C-g: g\in C\in \C(M)\}. $

       If $C_1,C_2\in \mathcal{D}_1$, then $M$ has a circuit $C_3$ such that $\{e_1,e_2\}\subseteq C_3\subseteq (C_1\cup C_2)-e$.
       Thus $g\not\in C_3$ so $C_3\in \C(M/g)$.
       Now suppose $C_1 \in \mathcal{D}_1$ and $C_2\in \mathcal{D}_2$. 
       Then $C_2\cup g$ is a circuit of $M$ and $M$ has a circuit $C_3$ such that $\{e_1,e_2\}\subseteq C_3\subseteq (C_1\cup C_2\cup g)-e.$
       If $g\in C_3$, then $C_3-g$ is a circuit of $M/g$ containing $\{e_1,e_2\}$ such that $C_3-g\subseteq (C_1\cup C_2)-e$.  
       If $g\not\in C_3$, then $C_3$ is a circuit of $M$ containing $\{e_1,e_2\}$ such that $C_3\subseteq (C_1\cup C_2)-e$.
       Similarly, if $C_1,C_2\in \mathcal{D}_2$, then $M$ has a circuit $C_3$ such that $C_3\subseteq (C_1\cup C_2\cup g)-e$ and $\{e_1,e_2\}\subseteq C_3$. 
       Either $C_3$ or $C_3-g$ is a circuit of $M/g$ contained in $(C_1\cup C_2)-e$ and containing $\{e_1,e_2\}$. 
    \end{proof}

The next three lemmas will be used to prove the main results. 
\begin{lemma}\label{new}
   Let $M$ be a connected matroid with ground set $D_1\cup D_2\cup e$ where $D_1$ and $D_2$ are skew circuits of $M$. Then either $M$ has a $2$-cocircuit avoiding $e$, or $M\cong S((U_{k-2,k};e),(U_{l-2,l};e))$ for some integers $k$ and $l$ exceeding two. 
\end{lemma}
\begin{proof}
    Since $M\del e= M| D_1\oplus M|D_2$, it follows that $M$ is equal to $S((M_1;e),(M_2;e))$ where $M_1=M/D_2$ and $M_2=M/D_1$.
    Evidently, $r(M)=r(D_1)+r(D_2)=|D_1|+|D_2|-2$, so $r(M^*)=3$ and $M^*$ has $D_1\cup e$ and $D_2\cup e$ as hyperplanes.
    Thus either $M^*$ has a 2-circuit avoiding $e$, or $M^*\cong P((U_{2,k};e),(U_{2,l};e))$ for some integers $k$ and $l$ exceeding two. 
    The lemma follows immediately. 
\end{proof}

The proof of the next lemma uses the well-known fact (see, for example, \cite[Exercise 2.1.7]{ox}) that $\{x,y\}$ is a circuit of a connected matroid if and only if every circuit that contains $x$ also contains $y$.

\begin{lemma}\label{singlecct}
    Let $S((M_1;e),(M_2;e))$ be a connected matroid $M$ having a circuit contained in $E(M_2)-e$. Then $M$ has $S((M_1;e),(U_{k-2,k};e))$ as a series minor for some $k$ exceeding two. Moreover, $E(U_{k-2,k})-e$ is skew to $E(M_1)-e$ in $M$. 
\end{lemma}

\begin{proof}
    Clearly, if $M$ has $S((M_1;e),(U_{k-2,k};e))$ as a series minor, then $E(U_{k-2,k})-e$ is skew to $E(M_1)-e$ since $S((M_1;e),(U_{k-2,k};e))\del e=(M_1\del e)\oplus (U_{k-2,k}\del e)$.  
    To see that $M$ has $S((M_1;e),(U_{k-2,k};e))$ as a series minor, let $M$ be a minimal counterexample.
    Let $C$ be a circuit contained in $E(M_2)-e$.
    Choose a circuit $D$ of $M_2$ that contains $e$ and meets $C$ so that $|D-C|$ is a minimum. 
    Then $E(M_2)=D\cup C$, otherwise $S((M_1;e),(M_2|(D\cup C);e))$ is a proper series minor of $M$, contradicting the minimality of  $M$.  
    
    We now show that $D-C$ is a series class of $M_2$. 
    Take $x\in (D-C)-e$. Suppose $M_2$ has a circuit $D'$ containing $e$ and avoiding $x$. 
    Then $D'\subseteq (D-x)\cup C$, so $|D'-C|<|D-C|$.
    Therefore $x$ is in every circuit of $M_2$ containing $e$, so $\{x,e\}$ is a cocircuit of $M_2$. 
    Thus $D-C$ is a series class of $M_2$. 

Since $M$ is series-minor-minimal, it follows that $E(M_2)\cap (D-C)=\{e\}$. 
Then $r(M_2)=r(C)=|C|-1$ so $r^*(M_2)=2$. 
Since $M_2^*$ is connected of rank two having $\{e\}$ as a flat, it follows that $M_2^*\cong U_{2,k}$ for some $k$ exceeding two.
Note that $M_2^*$ has no nontrivial parallel class, since $M_2^*$ is parallel-minor-minimal. 
Hence $M_2\cong U_{k-2,k}$, so $M\cong S((M_1;e),(U_{k-2,k},e))$, a contradiction.   
\end{proof}

For the next result, recall that $N_5$ is the series connection of two copies of $U_{1,3}$. 

\begin{figure}[ht]
    \centering
\begin{tikzpicture}[scale=1.26]
\coordinate (nw) at (0,1); \coordinate (ne) at (1,1);  \coordinate (s) at (0.5,0); 
    \draw[ultra thick] (s)to [bend left=15] (ne); \draw[ultra thick] (s)to [bend right=35] (ne);
    \draw[ultra thick] (s)to [bend left=35] (nw); \draw[ultra thick] (s)to [bend right=15] (nw);
    \draw[thick] (nw) -- (ne);

    \draw[fill=black] (nw) circle (2.2pt); 
    \draw[fill=black] (ne) circle (2.2pt);
    \draw[fill=black] (s) circle  (2.2pt); 
    \draw (.5,-0.5) node {$G_1$};
    \draw (.5, 1.15) node[scale=.75] {$e$};

    \begin{scope}[shift={(2,0)}]
    \coordinate (nw) at (0,1); \coordinate (ne) at (1,1);  
    \coordinate (s) at (0.5,0); 
    \draw[ultra thick] (s)to [bend left=15] (ne); 
    \draw[ultra thick] (s)to [bend right=30] (ne);
    \draw[ultra thick] (s)to [bend left=30] (nw); 
    \draw[ultra thick] (s)to [bend right=15] (nw);
    \draw[thick] (s)to [bend left=80] (nw);
    \draw[thick] (nw) -- (ne);

    \draw[fill=black] (nw) circle (2.2pt); 
    \draw[fill=black] (ne) circle (2.2pt);
    \draw[fill=black] (s) circle  (2.2pt); 
    \draw (.5,-0.5) node {$G_2$};
    \draw (-.15, .25) node[scale=.75] {$e$};
    \end{scope}

    \begin{scope}[shift={(3.9,0)}]  
    \coordinate (nw) at (0,1); \coordinate (ne) at (1,1);  
    \coordinate (s) at (0.5,0); \coordinate (e) at (1.5,.5);
    \draw[thick] (s)to [bend left=50] (nw); 
    \draw[ultra thick] (s)to [bend left=15] (e); 
    \draw[ultra thick] (s)to [bend right=30] (e);
    \draw[ultra thick] (s)--(ne); \draw[thick] (e)--(ne);
    \draw[ultra thick] (nw) -- (ne);
    \draw[ultra thick] (nw) -- (s);

    \draw[fill=black] (nw) circle (2.2pt);
    \draw[fill=black] (ne) circle (2.2pt);
    \draw[fill=black] (s) circle  (2.2pt);
    \draw[fill=black] (e) circle  (2.2pt);
    \draw (.5,-0.5) node {$G_3$};
       \draw (-0.14, .35) node[scale=.75] {$e$};
\end{scope}

\begin{scope}[shift={(6.2,0)}] 
    \coordinate (nw) at (0,1); 
    \coordinate (ne) at (1,1);
    \coordinate (sw) at (0,0);
    \coordinate (se) at (1,0);
    \coordinate (e) at (1.5,.5);
    \draw[ultra thick] (ne)--(nw)--(se)--(sw)--(ne); \draw[thick] (sw)--(ne);
    \draw[thick] (nw)--(sw);
    \draw[thick] (ne)--(e);
    \draw[ultra thick] (se)to [bend left=30] (e); \draw[ultra thick] (se)to [bend right=30] (e);

    \draw[fill=black] (nw) circle (2.2pt); 
    \draw[fill=black] (ne) circle (2.2pt);
    \draw[fill=black] (sw) circle (2.2pt); 
    \draw[fill=black] (se) circle (2.2pt);
    \draw[fill=black] (e) circle  (2.2pt);
    \draw (.5,-0.5) node {$G_4$};
    \draw (-.15, .5) node[scale=.75] {$e$};
\end{scope}

\begin{scope}[shift={(8.5,0)}] 
   \coordinate (n) at (0.5,1); \coordinate (e) at (1,.5);  \coordinate (s) at (0.5,0); \coordinate (w) at (0,.5); 
   
    \draw[thick] (n)--(e);
    \draw[thick] (n)--(s); 
    \draw[thick] (n)--(w); 
    \draw[ultra thick] (s)to [bend left=20] (e);
    \draw[ultra thick] (s)to [bend right=35] (e);
    \draw[ultra thick] (s)to [bend left=35] (w); 
    \draw[ultra thick] (s)to [bend right=20] (w);
    
    \draw[fill=black] (n) circle (2.2pt);
    \draw[fill=black] (e) circle (2.2pt);
    \draw[fill=black] (s) circle (2.2pt); 
    \draw[fill=black] (w) circle (2.2pt); 
    \draw (.5,-0.5) node {$G_5$};
    \draw (.39, .55) node[scale=.75] {$e$};
\end{scope}
\end{tikzpicture}
\caption{ }\label{2ske} 
\end{figure}
\begin{lemma}\label{2skewlem}
    Let $M$ be a connected binary matroid containing an element $e$ and a pair of skew circuits both avoiding $e$. Then, for some $i$ in $\{1,2,3,4,5\}$, $M$ contains $M(G_i)$ as series minor, where $G_i$ is one of the five graphs in Figure \ref{2ske}. Each such graphic matroid $M(G_i)$ has a unique pair of skew circuits avoiding $e$, indicated in bold in Figure \ref{2ske}.  
\end{lemma}
\begin{proof}
    Let $M$ be a minimal counterexample. 
    If $E(M)=C_1\cup C_2\cup \{e\}$, then $M=M(G_1)\cong N_5$, since $M\del e=M|(C_1\cup C_2)=M|C_1\oplus M|C_2$, and $M$ has no nontrivial series class. 
    Now suppose $M$ has an element $f$ contained in $E(M)-(C_1\cup C_2\cup e)$ such that $M\del f$ is disconnected. 
    Thus $M=S((M_1;f),(M_2;f))$, where $M_1$ and $M_2$ are nonempty, connected binary matroids.  
     
    Suppose first that $C_1$ and $C_2$ are circuits in $M_1$ and that $e$ is an element of $M_2$. Take a circuit $D$ of $M_2$ containing $\{e,f\}$. 
    Then $E(M_2)=D$ by the minimality of $M$. 
    Further, $\{e,f\}$ is a 2-cocircuit of $M$, a contradiction. 
    We may now assume that $C_1\cup e\subseteq E(M_1)$ and $C_2\subseteq E(M_2)$. 
    Since $M$ is binary, it follows from Lemma \ref{singlecct} that $M=S((M_1;f),(U_{1,3};f))$.
    
    Suppose $E(M_1)=C_1\cup \{e,f\}$. 
    Then $r(M_1)=r(C_1)$, otherwise $\{e,f\}$ is a cocircuit of $M$. 
    Hence $r^*(M_1)=3$. 
    It follows that $M_1^*$ is a rank-3 binary matroid containing $C_1$ as a cocircuit and $\{e,f\}$ as a hyperplane. 
    Hence $|C_1|\in \{2,3,4\}$ and $r(M_1)\in \{1,2,3\}$.  
    For each possible value of $r(M_1)$, there is a unique matroid satisfying these conditions, namely $U_{1,4}$ containing $C_1$ as a 2-circuit, $N_5$ containing $C_1$ as a 3-circuit, or $M(K_4)$ containing $C_1$ as a 4-circuit. 
    Hence $M_1$ is isomorphic to one of $U_{1,4}$, $N_5$, and $M(K_4)$. 
    Since $M=S((M_1;f),(U_{1,3};f))$, it follows that $M$ is isomorphic to $M(G)$ for some $G\in \{G_2,G_3,G_4\}$.  

    We may now assume that $M_1$ has an element $g\in E(M_1)-(C_1\cup \{e,f\})$. 
    Then $M_1\del g$ is disconnected, so $M_1=S((N_1;g),(N_2;g))$, where $N_1$ and $N_2$ are connected, binary matroids.  
    If $g$ is in a series pair in $M_1$, then $g$ is in a series pair of $M$, a contradiction. 
    It follows that we may assume that $\{e,f\}\subseteq E(N_1)$, while $C_1\subseteq E(N_2)$. 
    By Lemma \ref{singlecct}, we have that $M_1=S((N_1;g),(U_{1,3};g)$. 
    
    Suppose $N_1$ has a circuit $C$ avoiding $\{e,f\}$.
    Then $N_1\del x$ is disconnected for all $x$ in $C-g$.
    It follows by \cite[Lemma 2.3]{ox81} (see also \cite[Lemma~4.3.10]{ox}) that $C-g$ contains a $2$-cocircuit of $N_1$. 
    As this is a contradiction, every circuit of $N_1$ must meet $\{e,f\}$. 
    It follows that $\{e,f\}$ contains a cobasis of $N_1$, so $r(N_1^*)=1$, or $r(N_1^*)=2$. 
    If $r(N_1^*)=1$, then $N_1$ is a circuit and $\{e,f,g\}$ is a series class of $M_1$, a contradiction. 
    Thus $r(N_1^*)=2$. 
    Since $N_1^*$ is binary and parallel-minor minimal, $N_1^*\cong U_{2,3}$.  
    Therefore $M_1=S((U_{1,3};g),(U_{1,3};g)\cong N_5$, containing 2-circuits $C_1$ and $\{e,f\}$. 
    We conclude that $M\cong M(G_5)$.
\end{proof}

Note that the graphs $G_3$ and $G_5$ are isomorphic up to a relabeling of the special element $e$ that is specified in the statement of Lemma~\ref{2skewlem}. The unique pair of skew circuits avoiding $e$ in $M(G_3)$ have sizes two and three, while, in $M(G_5)$, the unique pair of skew circuits both have size two. 

\section{Proof of the Main Theorem}
In this section, we prove Theorem \ref{main} by showing that (iii) implies (ii), that (ii) implies (i), and that (i) implies (iii).    
The equivalence of (ii) and (iv) was proved in \cite[Theorem 1.1]{ub}.

\begin{proof}[Proof of Theorem \ref{main}.]

To see that (iii) implies (ii), let $M$ be a connected matroid satisfying (iii) and suppose that $M$ has $C_1$ and $C_2$ as skew circuits but that no proper connected series minor of $M$ has a pair of skew circuits.  
Let $D$ be a circuit of $M$ that intersects both $C_1$ and $C_2$ such that $|D-(C_1\cup C_2)|$ is minimal. 
Since $M|(C_1\cup C_2\cup D)$ is connected and $C_1$ and $C_2$ are skew in this restriction, $M=M|(C_1\cup C_2\cup D)$.  
We will show that $|D-(C_1\cup C_2)|=1$. 
Suppose $\{f,g\}\subseteq D-(C_1\cup C_2)$. 
We shall show that $f$ and $g$ are in series in $M$. 
Assume they are not. 
Then $M$ has a circuit $K$ that contains $f$ but not $g$. By the choice of $D$, the circuit $K$ cannot meet both $C_1$ and $C_2$. 
Since $K$ is not a proper subset of $D$, we may assume that $K\cap (C_1-D)\neq \emptyset$ and $K\cap C_2=\emptyset$. 
Take $h$ in $D\cap C_2$.
Then $M$ has a circuit $C_3$ such that $h\in C_3\subseteq (D\cup K)-f$. 
Because $C_3$ is not a proper subset of $D$, there is an element $k$ in $C_3-D$. 
Thus $C_3$ meets both $C_1$ and $C_2$ but avoids $f$, a contradiction to the choice of $D$. 
We conclude that $f$ and $g$ are in series in $M$. 
Then $M/g$ is a connected series minor of $M$ having $C_1$ and $C_2$ as skew circuits. 
This contradiction to the choice of $M$ implies that $|D-(C_1\cup C_2)|=1$. 
Let $D-(C_1\cup C_2)=\{e\}$. 
Then since $M$ has no 2-cocircuits, it follows from Lemma \ref{new} that $M\cong S(U_{k-2,k},U_{l-2,l})$ for some integers $k$ and $l$ exceeding two, which contradicts (iii). 
Thus (iii) implies (ii). 

\qquad

To see that (ii) implies (i), let $M$ be a connected matroid satisfying (ii) but not (i). 
Let $C_1$ and $C_2$ be distinct circuits of $M$ with $e\in C_1\cap C_2$, $e_1 \in C_1-C_2$, and $e_2\in C_2-C_1$. 
Then $M$ has no circuit $D$ such that $\{e_1,e_2\}\subseteq D\subseteq (C_1\cup C_2)-e$. 
For each $i\in \{1,2\}$, there is a circuit $D_i$ of $M$ such that $e_i\in D_i$ and $D_i\subseteq (C_1\cup C_2)-e$. 
By assumption, $D_1$ and $D_2$ are not skew. 
Then $M|(D_1\cup D_2)$ is connected, so $M$ has a circuit $D$ containing $\{e_1,e_2\}$ with $D\subseteq D_1\cup D_2$.
It follows that $\{e_1,e_2\}\subseteq D\subseteq (C_1\cup C_2)-e$, a contradiction. 
We conclude that (ii) implies (i). 

\qquad

Now assume that $M$ satisfies (i) but not (iii). 
Let $N$ be the series minor of $M$ that is isomorphic to $S(N_1,N_2)$, with $N_1\cong U_{k-2,k}$ and $N_2\cong U_{l-2,l}$ for some $k$ and $l$ exceeding two. 
Then $N$ satisfies SSCE by Lemma \ref{closed}.  
Let $p$ be the basepoint of the series connection $N$. 
Choose distinct circuits $C_1$ and $C_2$ of $N_1$, and distinct circuits $D_1$ and $D_2$ of $N_2$ so that  all of $C_1,C_2,D_1,$ and $D_2$ contain $p$. 
Let $e_1\in C_1-C_2$ and $e_2\in D_2-D_1$. 
Then $N$ has circuits $K_1$ and $K_2$ such that $K_1=C_1\cup D_1$ and $K_2=C_2\cup D_2$. 
Clearly $e_1\in K_1-K_2$ and $e_2\in K_2-K_1$, while $p\in K_1\cap K_2$. 
Since $e_1\in E(N_1)$ and $e_2\in E(N_2)$, every circuit of $N$ containing $\{e_1,e_2\}$ must also contain $p$. 
Hence $N$ does not satisfy SSCE, a contradiction.
Thus (i) implies (iii). We conclude that the theorem holds. 
\end{proof}

\section{Connected binary matroids with three skew circuits}
The goal of this section is to complete the proof of Theorem \ref{3thm}. The core of this proof is contained in Lemmas \ref{singlecct} and \ref{2skewlem}.  After giving this proof, we consider potential extensions of this theorem. 

\begin{proof}[Proof of Theorem \ref{3thm}.]
    Let $M$ be a minimal counterexample and let $C_1,C_2,$ and $C_3$ be skew circuits in $M$. 
    Since $M$ is connected, there is an element $e$ in $E(M)-(C_1\cup C_2\cup C_3)$. 
    Then $M\del e$ is disconnected, so $M=S((M_1;e),(M_2;e))$; that is, $M$ is a series connection of connected matroids $M_1$ and $M_2$ across the basepoint $e$. 

    Suppose first that $\{C_1,C_2,C_3\}\subseteq \C(M_1)$. 
    Take a circuit $C$ of $M$ containing $e$ and let $N=M|(E(M_1)\cup C)$. 
    The matroid $N$ has $C\cap E(M_2)$ contained in a series class. 
    Now $N/((C\cap E(M_2))-e)$ is a connected, binary, proper series minor of $M$ having $C_1,C_2,$ and $C_3$ as skew circuits. 
    Thus we contradict the minimality of $M$. 
        
    We may now assume that $C_1,C_2\in \C(M_1)$ and $C_{3}\in \C(M_2)$.
    Since $M$ is binary, by Lemma \ref{singlecct}, $M=S((M_1;e),(U_{1,3};e))$ with $E(U_{1,3})-e$ skew to $E(M_1)-e$. 
    By Lemma \ref{2skewlem}, the matroid $M_1$ is isomorphic to the cycle matroid of one of the graphs pictured in Figure \ref{2ske}. 
    Hence $M\cong S((M(G_i);e),(U_{1,3};e))\cong  M(L_i)$ for some $i$ in $\{1,2,3,4,5\}$.  
\end{proof}

The techniques used in the proof of Theorem \ref{3thm} can be extended to prove analagous results for connected binary matroids containing $k$ skew circuits, for $k\ge 4$. As one may gather from the proof of Theorem \ref{3thm}, the number of cases required to obtain an exhaustive list of connected, binary, series-minor-minimal matroids containing $k$ skew circuits becomes unmanageably large as $k$ increases.  

To see why an analagous result for non-binary matroids is not included in this paper, let $M'$ be isomorphic to the direct sum of $k$ circuits. Let $M$ be the matroid obtained by freely adding an element $e$ to $M'$. Then $M$ is connected and non-binary, containing $k$ skew circuits. The addition of $e$ has the effect of turning every 2-cocircuit of $M'$ into a 3-cocircuit of $M$ containing $e$. Thus $M$ contains no connected series minor containing $k$ skew circuits.       

\section{A new circuit axiom system}
The symmetric strong circuit elimination property does not hold for all matroids, and is therefore not equivalent to the weak and strong circuit elimination axioms, \textbf{(C3)} and \textbf{(C3)$'$} in \cite[pp.9 and 29]{ox}. By adding an additional hypothesis to the definition of SSCE presented in Section 1, we are able to provide a symmetric variant of the well-known circuit elimination axioms.

\begin{lemma} 
\label{cct}
The set ${\mathscr C}$  of circuits of a matroid $M$  obeys the following.
 \smallskip
 
\noindent{\bf (C3)$''$}  Let $C_1$ and $C_2$ be members of ${\mathscr C}$  with $e_1 \in C_1 - C_2$ and $e_2 \in C_2 - C_1$. If $e \in C_1 \cap C_2$ and $(C_1 - e_1) \cup (C_2 - e_2)$ contains no member of ${\mathscr C}$, then ${\mathscr C}$ contains a member $C_3$ such that $\{e_1,e_2\} \subseteq C_3 \subseteq (C_1 \cup C_2) - e$.  

Furthermore, $C_3$ is the unique circuit of $M$ contained in $(C_1 \cup C_2) - e$.
\end{lemma}

\begin{proof} 
Certainly $(C_1 \cup C_2) - e$ is dependent. Let $C_3$ be a circuit contained in this set. We shall show first that $\{e_1,e_2\} \subseteq C_3$. As 
$(C_1 - e_1) \cup (C_2 - e_2)$ is independent, we may assume that $e_1 \in C_3$. Suppose $e_2 \not\in C_3$. Then $e_1 \in C_1 \cap C_3$ and $e \in C_3 - C_1$, so there is a circuit $C_4$ such that $C_4 \subseteq (C_1 \cup C_3) - e_1$. Thus $C_4 \subseteq (C_1 - e_1) \cup (C_2 - e_2)$, a contradiction. We deduce that $\{e_1,e_2\} \subseteq C_3$.

To see that $C_3$ is unique, suppose there is a second circuit $C_3'$ contained in $(C_1 \cup C_2) - e$. Then $e_1 \in C_3 \cap C_3'$, so $M$ has a circuit $C_5$ contained in $(C_3 \cup C_3') - e_1$. As $C_5$ is contained in $(C_1 \cup C_2) - e$, we deduce that $\{e_1,e_2\} \subseteq C_5$, a contradiction.  Hence $C_3$ is indeed unique.
\end{proof}

The following theorem seems to give a new axiom system for matroids in terms of their circuits. For example, it is absent from the two standard reference books for the subject ~\cite{ox, djaw} and also does not appear in Brylawski's encyclopedic appendix of matroid cryptomorphisms \cite{thb}. 

\begin{theorem}
\label{cct2}
A collection ${\mathscr C}$ of nonempty pairwise incomparable subsets of a finite set $E$ is the set of circuits of a matroid on $E$ if and only if ${\mathscr C}$ satisfies {\bf (C3)$''$}.
\end{theorem}

\begin{proof}
 By Lemma \ref{cct}, if ${\mathscr C}$ is the set of circuits of a matroid on $E$, then ${\mathscr C}$ satisfies {\bf (C3)$''$}. Conversely, assume ${\mathscr C}$ satisfies {\bf (C3)$''$}. Suppose $C_1$ and $C_2$ are  distinct members of  ${\mathscr C}$ with $e$ in $C_1 \cap C_2$. Assume that 
{\bf (C3)} fails for $(C_1,C_2,e)$ and that 
 $|C_1 \cup C_2|$ is a minimum among such triples. As the members of ${\mathscr C}$ are incomparable, there are elements $e_1$ and $e_2$ of $C_1 - C_2$ and $C_2 - C_1$, respectively. By {\bf (C3)$''$}, $(C_1 - e_1) \cup (C_2 - e_2)$ must contain a member $C_4$ of ${\mathscr C}$, so $e \in C_4$. Then $e \in C_1 \cap C_4$ and $|C_1 \cup C_4| \le |(C_1 \cup C_2) - e_2| < |C_1 \cup C_2|$, so $(C_1 \cup C_4) - e$, and hence  $(C_1 \cup C_2) - e$ contains a member of ${\mathscr C}$, a contradiction.
 \end{proof}
 
It is tempting to try to weaken {\bf (C3)$''$} to require only that $e_1 \in C_1$ and $e_2 \in C_2$. To see that this variant need not hold, consider the cycle matroid of the graph $K_{2,3}$  and let $C_1$ and $C_2$  be the circuits  $\{e_1, a, e,e_2\}$ and $\{b, c, e,e_2\}$.  Then $(C_1 - e_1) \cup (C_2 - e_2)$ does not contain a circuit. But, although $(C_1 \cup C_2) - e$ does contain a circuit, that circuit does not contain $e_2$.

\end{document}